\documentclass[reqno,a4paper,12pt]{amsart}
 
\usepackage{latexsym, amsmath, amssymb, amsthm, a4}
\usepackage{graphicx}
\usepackage{subcaption}

\usepackage{amsfonts, bbold, bbm, marvosym}

\usepackage{mathrsfs}

\usepackage[mathscr]{eucal}

\newcommand{\g}{\gamma}

\newcommand{\de}{\delta}
\newcommand{\bd}{\pmb{\delta}}

\newcommand{\bD}{\pmb{\Delta}}
\newcommand{\e}{\epsilon}

\newcommand{\ka}{\kappa}

\newcommand{\la}{\lambda}

\newcommand{\m}{\mu}

\newcommand{\n}{\nu}
\newcommand{\x}{\xi}

\newcommand{\ro}{\rho}

\newcommand{\s}{\sigma}
\newcommand{\Si}{\Sigma}

\newcommand{\f}{\phi}

\newcommand{\F}{\Phi}

\newcommand{\Om}{\Omega}

\newcommand{\U}{\Upsilon}

\newcommand{\C}{{\mathbb C}}
\newcommand{\R}{{\mathbb R}}

\newcommand{\N}{{\mathbb{N}}}

\newcommand{\Ht}{\mbox{\texttt{H}}}

\newcommand{\ab}{{\mathbf a}}

\newcommand{\cb}{{\mathbf c}}

\newcommand{\nb}{{\mathbf n}}

\newcommand{\rb}{{\mathbf r}}

\newcommand{\tb}{{\mathbf t}}

\newcommand{\Ab}{{\mathbf A}}

\newcommand{\Bb}{{\mathbf B}}

\newcommand{\Cb}{{\mathbf C}}

\newcommand{\Eb}{{\mathbf E}}

\newcommand{\Hb}{{\mathbf H}}

\newcommand{\Nb}{{\mathbf N}}

\newcommand{\Sbb}{{\mathbf S}}
\newcommand{\Tb}{{\mathbf T}}

\newcommand{\Vb}{{\mathbf V}}

\newcommand{\Zb}{{\mathbf Z}}

\newcommand{\Vt}{\tilde{V}}

\newcommand{\AF}{\mathfrak A}

\newcommand{\Ec}{{\mathcal E}}

\newcommand{\Gc}{{\mathcal G}}
\newcommand{\Hc}{{\mathcal H}}

\newcommand{\Kc}{{\mathcal K}}
\newcommand{\Lc}{{\mathcal L}}

\newcommand{\Xc}{{\mathcal X}}

\newcommand{\dist}{{\rm dist}\,}

\newcommand{\supp}{\hbox{{\rm supp}}\,}

\newcommand{\W}{W}

\DeclareMathOperator{\re}{{\rm Re}\,}

\newcommand{\diam}{\operatorname{diam\,}}

\newcommand{\Tr}{\operatorname{Tr\,}}

\newcommand{\Sib}{\pmb{\pmb{\Si}}}

\newtheorem{thm}{Theorem}[section]
\newtheorem{cor}[thm]{Corollary}
\newtheorem{lem}[thm]{Lemma}

\theoremstyle{definition}

\newtheorem*{rem}{Remark}

\numberwithin{equation}{section}

\begin{document}

\title[Matrix operators]{Spectral asymptotics and estimates for matrix Birman-Schwinger operators with singular measures}

\author{Grigori Rozenblum}
\address{Chalmers Univ. of Technology, Gothenburg, Sweden }
\email{$\mathrm{grigori@chalmers.se}$}

\author{Grigory Tashchiyan${}^{\dag}$}
\address{St.Petersburg University for Telecommunications}
\begin{abstract} We consider operators of the form $\Tb=\Ab^*(V\mu)\Ab$ in $\R^\Nb$, where $\Ab$ is a pseudodifferential operator of order $-l$, $\mu$ is a compactly supported singular measure, order $s>0$ Ahlfors-regular, and $V$ is a weight function on the support of $\mu$. The  scalar type operator $\Ab$ and the weight function $V$ are supposed to be $m\times m$ matrix valued. We establish Weyl type asymptotic formulas for singular numbers and    eigenvalues of $\Tb$ for $\mu$ being the natural measure on a compact Lipschitz surface. For a general  Ahlfors-regular measure $\m$, we prove that the previously found upper spectral estimates are order sharp. 
\end{abstract}

\maketitle


\section{Introduction}\label{intro}
In recent years, there was an essential progress in the study of spectral properties of differential and pseudodifferential operators in unusual, often singular environments. This concerns, in particular,  Birman-Schwinger type operators of the form
 \begin{equation}\label{BS operator}
 \Tb=\Tb(\Ab,\Vb)=\Ab^* \Vb\Ab,
 \end{equation}
 where $\Ab$ is a negative order pseudodifferential operator and $\Vb$ is a properly defined 'weight'. Initially, in pathbreaking papers by M.Birman and M.Solomyak, see  \cite{BSMMO}, eigenvalue estimates and asymptotic formulas were obtained for  $\Vb$ being the operator of multiplication by a function $V$ in a certain class of integrable functions, while $\Ab$ is the resolvent of an elliptic operator.  In recent developments, $\Ab$  can be the resolvent of a subelliptic operator, see, e.g.,  \cite{PongeHeis}, \cite{Ponge.matrix}, \cite{MSZ}, \cite{SZ}, \cite{SXZ23}. On the other hand, a series of results on the eigenvalue distribution was obtained for $\Vb$ being, formally, a measure of the form $\Vb=V\mu,$ where $\mu$ is a singular measure  and $V$ is a real-valued weight function (density), see \cite{RTInt}, \cite{RT}. These results led to  a number of important developments in analysis and  noncommutative geometry, extending known constructions beyond their initial realm. A particular topic of application is the spectral study of the resolvent difference for singular perturbed Schr{\"o}dinger type equations, see \cite{GRRobin}, where $\mu$ is the Hausdorff measure on a Lipschitz surface $\Si$ of an arbitrary codimension and $V$ is a weight function on $\Si$.

In the scalar case, the proofs of the  results on the eigenvalues  asymptotics for such operators of the form \eqref{BS operator} in \cite{RT}, \cite{RS}, \cite{RTInt} in this singular setting are essentially based upon their reduction to an eigenvalue problem for potential type integral operators on $\Si$, for which spectral results had been previously established in \cite{RTLip1}, \cite{RTLip2}. This reduction, on its crucial step, requires that the weight function $V(x)$ is \emph{non-negative} on $\Si$. If this is not the case, but $V$ is still real-valued, a complicated approximative  sign-separation procedure is applied, enabling one to \emph{spatially} separate  positive and negative parts $V_{\pm}$ of $V$, so that the leading term in the asymptotics of positive eigenvalues of $\Tb$ is determined by the  positive part of $V$, ditto for the negative eigenvalues. 
 When passing  to operators \emph{acting   
in the space of vector-functions}, with $V(x)$ being  a Hermitian matrix-function on $\Si$, the reduction to integral operators goes through in the standard way for a sign-definite  matrix  $V(x)$. However, the above  sign-separation procedure breaks down if the matrix weight $V(x)$ is not sign-definite. Therefore, we failed in \cite{RT}, \cite{RS} to justify  the asymptotics of eigenvalues of $\Tb$ for  this case. Another important problem, where the reduction to integral operator fails, is the question on the asymptotics of singular numbers of the operator $\Tb$, in the case when the weight $V(x)$ is a non-Hermitian matrix function (a nonreal-valued function, in the scalar case).

In  the present paper we address these two questions. We propose an alternative sign-separation procedure, which enables us to find the separate asymptotics of positive and negative eigenvalues for a Hermitian non-signdefinite matrix weight $V(x)$ and the asymptotics of singular values for a non-Hermitian matrix weight function. Our approach uses a perturbation theorem in an abstract setting; this theorem may turn out to be useful for some other spectral problems as well.

Asymptotic formulas for operators \eqref{BS operator} are obtained for measure $\mu$ being the Hausdorff measure on a Lipschitz surface (and on certain more complicated sets.) This restricts, in particular, their application to   the case of sets of an integer Hausdorff dimension only. For a fractional dimension, it is known that it may happen that the eigenvalues (singular numbers) do not have power asymptotics. Our general abstract result enables us to prove that, under some mild regularity conditions, our upper spectral estimates are order sharp, they are supported by lower estimated of the same power order.

\textbf{
Statement by G.Rozenblum.} This paper continues a series of  joint studies with Grigory (Grisha) Tashchiyan, starting with \cite{RTLip1}. We were close friends since our student days, in 1960-s; after that, we both were PhD students under the inspiring supervision by Prof. M.Z. Solomyak. Co-operation with Grisha was quite productive, we exchanged ideas and assisted each other in dealing with technical difficulties, notwithstanding being at times rather far apart geographically. In spite of a progressing illness, Grisha continued thinking about this project and kept producing nice ideas. Unfortunately, he didn't live to see the final result of our work.

\section{Setting. Sign-separation}\label{setting}
\subsection{The sign problem} In the process of the study of spectral estimates and asymptotics, one considers rather  often operators involving some function acting as weight. The first and quite typical example is the weighted Laplace eigenvalue problem 
\begin{equation}\label{weight.1}
  -\Delta u =\lambda V(x) u,
\end{equation}
 say, with Dirichlet boundary conditions, see \cite{BSFaa}. Since long ago, it was known that in the case when the real-valued weight function $V$ is sign-definite, say, non-negative, the problem \eqref{weight.1} possesses only a sequence of positive eigenvalues, with asymptotic behavior determined by $V$. What is more interesting is that if $V$ is \emph{not} sign-definite, there are two sequences of eigenvalues, positive and negative ones, moreover, the asymptotics of positive eigenvalues is determined exclusively by the positive part of $V$, while the asymptotics of negative ones is determined by its negative part, with both having the same power order. Further on, this effect of sign-separation  was observed for quite a lot of non-signdefinite spectral problems including  the ones where this non-signdefiniteness is present in the leading term, as, e.g., in \cite{BSnegative}, for negative order pseudodifferential operators with matrix symbol. Although such sign-separation  property looks quite natural, its proving is often rather nontrivial. Moreover, there are examples of spectral problems where visibly nontrivial  negative part of the weight produces, nevertheless, only finitely many negative eigenvalues, while the asymptotics of eigenvalues is determined only by the positive part of the weight, see, e.g.,  \cite{BM}, \cite{PR},  \cite{GRlower}.

  There are several standard ways of proving eigenvalue asymptotics in the non-signdefinite case. First of all, if it is possible, the problem is solved by means of  a reduction to the fundamental result by M.Birman and M.Solomyak \cite{BSnegative} about the spectral asymptotics of negative order 
pseudodifferential operators, where the principal (matrix) symbol is not   supposed to be sign-definite. The other way consists in reducing algebraically the problem to the sign-definite one, where the asymptotic results are already known. This was done in quite a number of publications,  \cite{Ponge.matrix} being, probably, the latest one. Sometimes, when the operator in question is not a pseudodifferential one, such reduction is not available. Here, some  localization may help. Say, in \cite{RS}, \cite{RT},  \cite{RTInt}, using order sharp spectral estimates,  obtained aforehand, the positive and negative parts of the weight were spatially separated by means of a proper perturbation, and then the results on sign-definite weights could be applied. In these papers, however, since the weight was fairly singular, the Birman-Solomyak theorems, as in \cite{BSnegative}, could not be applied directly, however we had already at disposal the results by the authors, on potential type operators on Lipschitz surfaces, see \cite{RTLip1}, \cite{RTLip2}, and to such operators, the localized problem could be reduced. Unfortunately, this latter approach cannot be applied for non-signdefinite \emph{matrix} problems, since  the spatial separation is not possible here. An alternative approach was proposed recently in  \cite{MSZLMF}, \cite{SZFA}. There, for a non-signdefinite matrix weight, an algebraic trick was applied, enabling one still to separate \emph{spectrally}, but not spatially, the positive and negative parts of the operator;  this approach, however, is  based upon commutator estimates valid for  smooth weights only.  In this paper, we modify this approach, enabling one to handle singular weights.
 \subsection{Spectral characteristics of compact operators}
 Further on, for a \emph{self-adjoint} compact operator $\Tb$, we denote by $\Tb_{\pm}$ the positive, resp., negative part of $\Tb$; this means
\begin{equation*}
  \Tb_\pm=\frac12(|\Tb|\pm\Tb),
\end{equation*}
where $|\Tb|$ is the absolute value of $\Tb$ in the sense of spectral theory: for any bounded operator $\Tb$, as usual, $|\Tb|=(\Tb^*\Tb)^{\frac12}$.

 For a given compact operator $\Tb$, we denote by $n(\la, \Tb), \, \la>0$ the counting function for the singular values $\s_j(\Tb)$ of $\Tb$:
 \begin{equation}\label{n(mu)}
 n(\la,\Tb)=\#\{j:\s_j(\Tb)>\la\};
 \end{equation}
  if the operator $\Tb$ is self-adjoint, the counting functions of its positive, resp. negative eigenvalues $\la_{\pm}(\Tb)$ are introduced,
  \begin{equation}\label{n pm}
    n_{\pm}(\la,\Tb)=\#\{j:\pm\la_j^{\pm}(\Tb)>\la\};
  \end{equation}
the singular- and eigenvalues are counted in \eqref{n(mu)}, \eqref{n pm} respecting their multiplicity. Of course, for a self-adjoint operator, 
\begin{equation*}
   n(\la,\Tb)=n_+(\la,\Tb)+n_-(\la,\Tb)=n_+(\la, |\Tb|).
\end{equation*}
The traditional topic in operator theory is the study of the behavior of these quantities as $\la\to 0$. To describe asymptotic estimates, for a fixed exponent $\theta,$ we introduce the notations 
 \begin{equation*}
\begin{array}{l}{\bD^{\theta}(\Tb)}\\{\bd^{\theta}(\Tb)}\end{array}
=\begin{array}{l}{\limsup}\\{\liminf}\end{array}_ {\la\to 0}\la^{\theta}n(\la,\Tb),  
 \end{equation*}
similarly, for a self-adjoint operator $\Tb$, the notations $\bD_{\pm}^{\theta}(\Tb)$, $\bd_{\pm}^{\theta}(\Tb)$ are  introduced. If the upper ($\bD$) and lower ($\bd$) asymptotic coefficients coincide for some choice of the sign ${\pm}$, this means, when the eigenvalue asymptotics holds for this sign, we use the symbol $\nb^\theta(\Tb)$ for their common value, for example,
\begin{equation*}
  \nb^{\theta}_\pm(\Tb)=\bd^{\theta}_\pm(\Tb)=\bD^{\theta}_\pm(\Tb)=\lim_{\la\to 0}\la^{\theta}n_\pm(\la,\Tb).
\end{equation*}
Following the notations of the Russian school of spectral theory, we denote by $\Sib_\theta$ the weak Schatten ideal, this means the set of operators $\Tb$ satisfying $\bD_\theta (\Tb)<\infty$.
 The separable ideal $\Sib_\theta^\circ$ consists of operators satisfying $\bD_\theta(\Tb)=0.$
\subsection{Main results} According to the results of \cite{RS}, \cite{GRInt}, \cite{RT}, for an operator \eqref{BS operator} with $\Vb=V\mu,$ and an order $-l$ pseudodifferential operator $\Ab$ (in the next section, we give a detailed definition of  the operator \eqref{BS operator} in this singular setting)  with \emph{scalar} real-valued density $V$ and  measure $\m$ being the Hausdorff measure on a Lipschitz surface $\Si$ of dimension $s$, the asymptotics of the eigenvalues is valid:
\begin{equation*}
  \nb^{\theta}_{\pm}(\Tb)=\Cb_{\pm}(\Ab,\Si, V), \, \theta=\frac{s}{s-\Nb+2l},
\end{equation*}
with coefficients $\Cb_{\pm}=\Cb_{\pm}(\Ab,\Si, V)$
calculated according to
\begin{equation}\label{Coeff.Scalar}
  \Cb_{\pm}=\int_{\Si}\ro_{\Ab}(x)V_{\pm}^{\theta}(x)\m(dx),
\end{equation}
 where $\ro_{\Ab}(x)$ is a density on $\Si$ determined by the principal symbol of operator $\Ab$.
 
 For a measure $\m$ and the exponent $\theta$, we introduce the notation  $L_{(\theta),\m}$, which  stands for $L_{\theta,\m}$ if $\theta>1$, for $L_{1,\m}$  if $\theta<1$, and, finally, for the Orlicz space $L_{\Psi,\m}$ corresponding to the function $\Psi(t)=(t+1)\log(t+1)-t$,  if $\theta=1$.  By $\|.\|_{(\theta),\m}$ we denote the norm in this space.

In our main results, we extend \eqref{Coeff.Scalar} to two new situations.
\begin{thm}\label{Main theorem pm} Let $\Ab$ be an order $-l$ scalar type pseudodifferential operator, of the form
\begin{equation*}
 \Ab=\Ab_0\otimes\Eb_m,
\end{equation*}
where $\Ab_0$ is a scalar pseudodifferential operator, $\Eb$ is the $m\times m$ unit matrix, and $V(x),\, x\in\Si$, is an $m\times m$ matrix density, $V\in L_{(\theta),\m}.$
Then for the singular numbers of operator \eqref{BS operator}, the asymptotic
formula holds
\begin{equation}\label{snumbers.as}
  \nb^{\theta}(\Tb)=\int_{\Si}\ro_{\Ab}(x)\Tr(|V(x)|^{\theta})\m(dx). 
\end{equation}
If, moreover, the matrix $V(x)$ is Hermitian, then
\begin{equation}\label{eigen.pm}
   \nb_{\pm}^{\theta}(\Tb)=\int_{\Si}\ro_{\Ab}(x)\Tr(V_{\pm}(x)^{\theta})\m(dx),
\end{equation}
with the same density $\ro_{\Ab}(x)$ as in the scalar case.
 \end{thm}

Theorem \ref{Main theorem pm} concerns only the Hausdorff  measure of an integer dimension on a Lipschitz surface. To consider more general measures, we recall the following definition, see, e.g., \cite{David}.

 Let $\m$ be a finite Borel measure in $\R^\Nb$ with compact support $\Si$. For $s\in(0,\Nb]$, we set $\m\in\AF^{s}$ if $\m$ satisfies 
 \begin{equation}\label{Ahl.def}
   \cb_- r^s\le \m(B(x,r))\le \cb_+r^s, \, x\in \supp(\m), \,r\le \diam\supp(\m), 
 \end{equation}
 where $B(x,r)$ is the closed ball of radius $r$ centered at $x$ and $\cb_-=\cb_-(\m)>0,$ $\cb_+=\cb_+(\m)<\infty$ are   some constants. Such measures are called Ahlfors $s$-regular. If only the left, resp., right inequality in \eqref{Ahl.def} holds, the measure is called lower-, resp., upper $s$-regular; the corresponding class is denoted by $\AF^s_-,$ resp., $\AF^s_+$. Note that upper regular measures may not contain point masses.  
In  particular, the Hausdorff measure on a Lipschitz surface of dimension $s$ belongs to $\AF^s$. There are examples (see, e.g., \cite{NaiSol}, \cite{SolVer}) showing that for general measures $\mu\in\AF^s$,  operator \eqref{BS operator} with $\Vb=V\mu,$ may fail to have a power eigenvalue asymptotics. For a scalar function $V>c>0$, two-sided power estimates for $n(\la,\Tb)$ were obtained for such measures, see \cite{Tr}, \cite{RT}, \cite{RS}. Here, in our second main theorem, we show that on the level of estimates, the spectral sign-separation effects for matrix densities $V(x)$, similar to the one in Theorem \ref{Main theorem pm}, hold for general measures  $\m\in\AF^s$ under some mild restrictions imposed  on the matrix density $V$. Namely, see Theorem \ref{Thm.lower} below for details, if $V\in L_{(\theta),\m},$ $\theta=\frac{s}{s-\Nb+2l}$, and $|V(x)|>c>0$ on $\Si\cap\Om$ for some open set $\Om\subset\R^\Nb$ with nonvoid intersectionwith $\Si$, under the assumption that the operator $\Ab$ is elliptic in  $\Om$, the two-sided singular numbers estimate 
\begin{equation*}
  0<\bd^{\theta}(\Tb)\le \bD^{\theta}(\Tb)<\infty
\end{equation*}
holds. For a Hermitian matrix density $V(x)$, if $V(x)>c>0$ for $x\in\Si\cap\Omega,$ the positive eigenvalues satisfy
\begin{equation*}
   0<\bd_+^{\theta}(\Tb)\le \bD_+^{\theta}(\Tb)<\infty.
\end{equation*}

 \section{Operators with singular measures}\label{Sect.measures}
   \subsection{The Birman-Schwinger type operators}
  In this paper, we consider  Birman-Schwinger type operators.  
  \begin{equation}\label{general form}
    \Tb=\Tb(\m,\Ab,W, U)=\Zb^*U\Zb,\, \Zb=W\g\Ab,
  \end{equation}
   where $\Ab$ is a pseudodifferential operator of order $l<0$, $\m$ is a singular compactly supported measure belonging to $\AF^s$; next, $W\ge0,$ $U$ are  $\m$-measurable $N\times N$-matrix functions, $U\in L_{\infty,\m}$. By $\g$ we denote the operator of restriction of functions in the Sobolev space $\Ht^l(\R^\N)$ to $\Si$.
   The definition of such operators with singular measures is discussed in detail in \cite{RT}, \cite{RTInt}. 
The trace operator
 $\g:\Ht^{l}(\R^\Nb)\to L_{2,\mu}$ is defined first on $\Ht^{l}(\R^\Nb)\cap C(\R^\Nb)$ and, if this operator is bounded as acting to $L_{2,\m}$, it can be  extended by continuity to the whole of $\Ht^l(\R^\Nb)$. The latter property, as explained in \cite{RTInt}, \cite{RT},  holds provided $2l>\Nb-s$. If this is the case, operator $\g\Ab$ is a bounded operator from $L_2(\R^\Nb)$ to $L_{2,\m}$.

In this paper, we restrict our considerations to the case when $\Ab$ is a scalar-type operator, this means that it has the form $\Ab=\Ab_0\otimes \Eb_m$, with a scalar pseudodifferential operator $\Ab_0.$ More general cases, when $\Ab$ is a truly matrix, not a scalar-type, operator, or even  for operators of a non-symmetric form $\Tb(\Ab_1,\Ab_2,W_1,W_2)=(\W_1\g\Ab_1)^*(W_2\g\Ab_2),$   are considerably more complicated and will be studied on a later occasion, with analysis, not that soft as here.
 About the weight matrix-function $V=W^*U W$, it is supposed that the composition $\Zb=W\g\Ab $ is a bounded, moreover, compact operator  from $L_2(\R^\Nb)$ to $L_{2,\m}$.

This condition is satisfied  provided 
   \begin{equation}\label{Vcondition}
     W^2\in L_{(\theta),\mu}, \, \theta=\frac{s}{s-\Nb+2l}, \, U\in L_{\infty,\m}
   \end{equation}
   (a detailed discussion of this condition, as well as references to sources, mainly to \cite{MazBook}, can be found in \cite{RS}, \cite{RT}, \cite{GRLT}.)  Further on, the measure $\m$ and the operator $\Ab$  will be usually fixed, and we will omit them in  notations, as long as this does not cause confusion.

The operator $(W\g\Ab)^*(W\g\Ab)$ is, of course, self-adjoint, so, possible non-self\-adjointness can be only caused here  by the matrix-function $U(x)$ in \eqref{general form} only. If, moreover $U$ is Hermitian, then $\Tb$ \eqref{general form} is self-adjoint. 

Operator $\Tb=\Tb(W, U)$ can be described by means of its quadratic form

\begin{equation}\label{QformDefinition}
  \tb[u]=\int \langle U(x)W(x)(\g\Ab u)(x), W(x)(\g\Ab u)(x)\rangle \mu(dx), 
\end{equation}
where angle brackets denote the inner product in $\C^{m}.$ The expression in
\eqref{QformDefinition} can be re-written as 
\begin{equation*}
  \tb[u]=\int \langle V(x)(\g\Ab u)(x),(\g\Ab u)(x)\rangle\mu(dx), \, V=W^*U W;
\end{equation*}
this means that the operator $\Ab$ is determined not by the weight matrix-functions $W(x), U(x)$ separately, but only by  their product, $V=W^*U W.$ Thus, we will leave it in the notation: $\Tb=\Tb(V)$. Note, also, that, conversely, any matrix-function $V(x)\in L_{(\theta),\m}$ can be factorized as $V(x)=W(x)^* U(x) W(x)$, with $W^2\in L_{(\theta),\m}$ and $U\in L_{\infty,\m}$.

For general operators of the form \eqref{general form}, we will study the behavior of singular values, $\s_j(\Tb)$. In the self-adjoint case,
$V=W^*UW,$ where $U$ is a \emph{Hermitian}  matrix function in $L_{\infty,\m}$, we will and consider also the eigenvalues $\pm\la_j^{\pm}(\Tb)$.
    \subsection{Spectral estimates}\label{Sub.estim}
   Main results in  \cite{RS}, \cite{RT}, which  we will use in our analysis, are the following.
   \begin{thm}\label{Est.Thm} Let the conditions \eqref{Vcondition} be satisfied. Then the operator $\Zb\equiv W\g\Ab:L_2(\R^\Nb)\to L_{2,\m}$ is bounded  and, moreover, its singular numbers satisfy 
     \begin{equation}\label{MainEstimate}
     \bD^{2\theta}(\Zb)\le C \|W^2\|_{(\theta),\m}^{\theta},
   \end{equation}
   where the constant $C$ depends only on the measure $\m$, the operator $\Ab$ and the parameters $\Nb,s,l$, but not on the weight function $W$. 
   \end{thm}
   \begin{cor}\label{cor.estimates}\begin{enumerate} \item[(i)] For $V=W^*UW$, the estimate \eqref{MainEstimate} takes the form     
     \begin{equation}\label{main.cor1}
       \bD^{\theta}(\Tb(V))\le C \|V\|_{(\theta),\m}^{\theta};
     \end{equation}
     \item[(ii)] If $V=V^*$ is a Hermitian matrix function, then
         \begin{equation}\label{main.cor2}
           \bD_{\pm}^{\theta}(\Tb(V))\le C \|V_{\pm}\|_{(\theta),\m}^{\theta},
         \end{equation}
         where $V_{\pm}(x)=|V(x)|\pm V(x)$,\, $|V(x)|=(V(x)^*V(x))^{\frac12}$.
\end{enumerate}
   \end{cor}
   \begin{proof}Statement (i) follows from the Ky Fan inequality for singular values of the product of compact operators. Statement (ii) follows from the variational principle for self-adjoint compact operators since $\pm\Tb(V)\le \Tb (\pm V)$.
\end{proof}  
     \begin{rem}1. In the sources cited above, estimates are given for scalar weight functions; they carry over to the matrix ones automatically, component-wise, see, e.g., explanations in \cite{RTInt}.\\
     2. In the above  papers, a considerably stronger versions of upper singular numbers estimates are stated, for counting functions themselves, $n(\la,\Tb)
\le C \la^{-\theta}\|V\|_{(\theta),\m}^\theta$, unlike \eqref{MainEstimate} here, where only the upper \emph{asymptotic} estimates are given. The asymptotic estimates in \eqref{MainEstimate} are sufficient for the applications in this paper.  \\
3. Under some relations between $s,l,\Nb, $ the conditions imposed on the measure $\m$ may be relaxed, in particular, by changing the requirement $\m\in\AF^s$ to $\m\in\AF^s_+$ or  $\m\in\AF^s_-$, cf. \cite{RT}
 and \cite{RS}.
     \end{rem}
     
\section{Proving the spectral asymptotics formulas}
 \subsection{Approximation approach}
The first step in the spectral studies of the operator \eqref{general form} consists in replacing the weight function $V$ by smooth ones. This is done following the general perturbation scheme developed in 1970-s by M.Birman and M.Solomyak. The asymptotic approximation lemma, established first in \cite{BSFaa}, can be formulated in our terns as follows:
\begin{lem}\label{BSlem} Let the given compact operator $\Tb\in\Sib_{\theta}$ can be approximated by operators $\Tb_{\e}$, so that $\bD^{\theta}(\Tb-\Tb_\e)\to 0$ as $\e\to 0$. Then
\begin{equation}\label{convergence}
\bD^{\theta}(\Tb)=\lim_{\e\to^0}\bD^{\theta}(\Tb_\e), \, \bd^{\theta}(\Tb)=\lim_{\e\to 0}\bd^{\theta}(\Tb_\e).
\end{equation}
In particular, if the limits on the right-hand sides in \eqref{convergence} are equal, then $\bD^{\theta}(\Tb)=\bd^{\theta}(\Tb)$; this means that \eqref{convergence}
 describes the power asymptotics of the singular values of $\Tb$. In the case when the operators $\Tb$ and $\Tb_{\e}$ are self-adjoint, then \eqref{convergence} also holds with $\bD_{\theta}, \bd_{\theta}$ replaced by $\bD_{\pm}^{\theta}, \bd_{\pm}^{\theta}$, so the above equality means the power asymptotics for positive and for negative eigenvalues of the limiting operator $\Tb$, separately. 
 \end{lem}
 In our setting, we are going to approximate the weight function $V$ by sufficiently regular ones $V_{\e}$, so that the  role of $\Tb$ in Lemma \ref{BSlem} is played by the  operator $\Tb(V)$, and the role of $\Tb_{\e}$ is played by $\Tb(V_\e)$. Under this concretization, the relation $\bD^{\theta}(\Tb-\Tb_\e)\to 0$ is granted  by the convergence $\|V-V_\e\|_{(\theta),\m}\to 0$, according to estimate \eqref{MainEstimate} in Theorem \ref{Est.Thm}. 

Therefore, when establishing asymptotics for singular numbers  or eigenvalues  of $\Tb(V)$, we need to construct the approximating densities $V_\e,$ such that 
\begin{enumerate}
  \item[(i)] Densities $V_\e$ approximate the density $V$ in the sense $\|V-V_\e\|_{(\theta),\m}\to 0,$ as $\e\to 0$, \\
      and\\
  \item[(ii)] For the operators $\Tb(V_\e)$, the singular numbers  asymptotics 
  \begin{equation*}
    \nb^{\theta}(\Tb(V_{\e}))=\F(V_\e),
  \end{equation*}
  resp. the eigenvalues asymptotics 
  \begin{equation*}
    \nb^{\theta}_{\pm}(\Tb(V_{\e}))=\F_{\pm}(V_\e),
  \end{equation*}
  are known, where $\F$, $\F_{\pm}$ are functionals over matrix functions, possessing the property of continuity in the norm  in $L_{(\theta),\m}$. 
\end{enumerate}
As shown in \cite{RT}, \cite{RS}, in order to grant  condition (ii), one needs that the approximating density $V_{\e}$ be defined not only on the support of the measure $\m$, but, at least, in some neighborhood of this support and be smooth there. The construction of such approximation is presented in the next subsection. 
  \subsection{Approximation}\label{Sect.approximation}    
Having a  density $V$ in $L_{(\theta),\m},$  thus defined initially only on $\Si=\supp(\m)$, we approximate it  in the sense of $L_{(\theta),\m}$ by smooth matrix functions, defined on $\R^\Nb$. In \cite{GRInt}, \cite{RT}, etc., such approximation was constructed for a Lipschitz surface $\Si$. Here, since we need a more general result, we use a different construction.
 
 \begin{lem}\label{lem.appr} Let $\m$ be a finite Borel measure with compact support $\Si\subset\R^{\Nb}$, $V\in L_{(\theta),\mu}$. We extend $V$ by zero outside $\Si$. Then, for any $\e>0$, there exists a function $V_\e\in C_0^\infty(\R^\Nb)$, such that 
 \begin{equation}\label{appr}
   \|V-V_\e\|_{L_{(\theta),\mu}}\le \e.
 \end{equation}
 Moreover, if the matrix-function $V$ is Hermitian, then $V_\e$ can be chosen to be Hermitian as well.
 \end{lem}
 \begin{proof}Essentially, the proof is contained in \cite{MakPodBook}, see there Theorem 13.3.3, or \cite{Tr}, Theorem 3.2. For completeness, we reproduce this proof here, with necessary modifications. Note that the case $\theta=1,$ this means, for $V$ in the Orlicz space, is formally not covered by this theorem, however the reasoning is, actually, the same. 

Note first, that it suffices to consider approximations of a scalar function, since for a matrix one, we can approximate each component separately.
 
 Recall that  a 'simple function' in $L_{(\theta),\mu}$ is a \emph{finite} sum of step functions,
 \begin{equation*}
   f=\sum_{j\le j_0}a_j\chi_{E_j},
 \end{equation*}
  where the sum is finite and $\chi_{E_j}$ are characteristic functions of disjoint $\mu$-measurable sets. Since the set of simple functions is dense in $L_{(\theta),\m}$ it suffices to show that the characteristic function of a $\mu$-measurable set $E$, $\m(E)<\infty$, can be approximated in the required way. For a given  $\e>0$, we fix an  open set $G$ and compact set $F$ so that $F\subset E\subset G$ and $\m(G\setminus F)<\e$, this is possible since the measure $\m$, being Borel,  is regular (in the sense of general measure theory, see, e.g. \cite{MakPodBook}, Theorem 13.3.2.) By the Tietze lemma, there exists a continuous function $\f$ on $\R^\Nb$ such that $0\le \f\le1$, $\f(x)=1$ on $F$ and $\f=0$ outside $G.$ Then, for $1\le p<\infty$,
  
  \begin{equation*}
    \|\f-\chi_E\|_p^p=\int_G |\f(x)-\chi_\e(x)|^p\mu(dx)\le \m(G\setminus F)<\e.
  \end{equation*}
  Similar calculation takes care of the Orlicz space estimate. It remains to approximate the continuous function $\f$ fy a $C_0^{\infty}(G)$ function $\psi$ so that $\|\f-\psi\|_{L_\infty(G)}<\e^{1/\theta}$.

The remark on Hermitian matrix $V(x), \, x\in\Si,$ is taken care of by the fact that if some matrix $V_\e$ approximates $V$, then $\re(V_\e)=\frac{V_\e+V_\e^*}{2}$ approximates  $V$ as well. 
 \end{proof}

 It is important to notice that if the function $V(x)$ is real and $V(x)>c>0$
 in $\Si\cap\Om$, where $\Om$ is an open set, then the approximating function $V_\e$ can be chosen satisfying $V\e>c/2$ in any proper subset $\Om'\subset\Om$.  
 \subsection{Sign separation in the scalar case and its breakdown  in the matrix case}\label{matrix breakdown} We explain here how the case of a \emph{real} \emph{scalar} non-signdefinite density $V$ was treated in \cite{RS}, \cite{RT}, \cite{GRInt}. After  having  constructed the approximating \emph{real} smooth function $V_{\e}(x)$, $x\in \R^{\Nb}$, we can construct a further approximation, with a special property. Namely, the function $R(x)$ was constructed
so that its positive part $R_+$ and its negative part $R_-$ are spatially separated: this means that $\dist(\supp(R_+),\supp(R_-))>0$. After that, we can introduce smooth cut-off function that equals 1 in a neighborhood of $\supp(R_+)$ and vanishes in a neighborhood of $\supp(R_-)$. Then, using the fact that the commutant of the operator $\Ab$ and such cut-off function is a pseudodifferential operator of order $-l-1$, we can show that the asymptotic characteristics  of the positive/negative  eigenvalues of the operator $\Tb(R)$ coincide with the corresponding characteristics of $\Tb(R_{\pm})$,
\begin{equation*}
  \bd^{\theta}_{\pm}(\Tb(R))= \bd^{\theta}_{\pm}(\Tb(R_{\pm})),  \bD^{\theta}_{\pm}(\Tb(R))= \bD^{\theta}_{\pm}(\Tb(R_{\pm})).
\end{equation*}
Thus, the problem was reduced to finding eigenvalue asymptotics for operators with sign-definite (say, nonnegative), $R$. All this goes through with success in the matrix case as well.

The last transformation of the problem consists in the reduction to a completely different operator, namely, to an integral operator on the set $\Si$, in the following way:  we write our operator $\Tb(R)$ as a product
\begin{equation}\label{T(R)}
\Tb(R)=(R^{\frac12}\g\Ab)^*(R^{\frac12}\g\Ab)=(\g R^{\frac12}\Ab)^*(\g R^{\frac12}\Ab),
\end{equation}
since the multiplication by the smooth function $R^{\frac12} $ and the restriction $\g$ commute.

And here we perform the main trick! The (non-negative) eigenvalues of the self-adjoint  operator $\Tb(R)$ in \eqref{T(R)} coincide with eigenvalues of the product of factors in  \eqref{T(R)} taken in the inverse order, 
\begin{equation*}
  \Sbb(R)=(\g R^{\frac12}\Ab)(\g R^{\frac12}\Ab)^*.
\end{equation*}
This operator $\Sbb(R)$ can be (after some deliberation) written as 

\begin{equation*}
  \Sbb(R)=\g(\g(R^{\frac12}\Ab^2 R^{\frac12}))^*.
\end{equation*}
Since any order $-2l$ pseudodifferential operator $\Ab^2$ is an integral operator in $\R^{\Nb}$ with the kernel $\Kc(x,y)$ having leading diagonal singularity
of order $2l-\Nb$ in $x-y$, or, in some cases, with a logarithmic factor, our operator $\Sbb(R)$ is the integral operator in $L_2(\Si,\m)$ with kernel 
$R^{\frac12}(x)\Kc(x,y)R^{\frac12}(y)$.

After this last reduction, we need only to look for situations, namely, measures $\m$, where the eigenvalue asymptotics for such integral operators has been already found. This kind of results were, in particular, obtained for $\mu$ being the Hausdorff measure on a compact Lipschitz surface of dimension $s$ in $\R^{\Nb}$, see \cite{RTLip1,RTLip2}. Some other measures are also admissible (see \cite{GRInt}, \cite{RTInt}). 

Now, let us consider the matrix case. Let the matrix $R(x)$ (we can suppose it smooth already) be non-negative on $\Si.$ Then the reasoning above goes through without any formal changes. However, if the matrix $R(x)$ is 
not sign-definite for $x$ on a set of positive $\m$-measure, no localization can spatially separate  positive and negative parts of $R(x)$.
One might have tried to perform the above reasoning in a non-symmetric form,
namely, to consider instead of \eqref{T(R)}, the representation
\begin{equation*}
  \Tb(R)=(\g R\Ab)^*(\g \Ab).
\end{equation*}
Then the same commutation  trick would lead to the operator $\Sbb'(R)$ instead of $\Sbb(R)$,

\begin{equation}\label{Sb'}
 \Sbb'(R)=(\g \Ab)(\g R\Ab)^*=\g(\g R\Ab^2)^*.
\end{equation}
Operator \eqref{Sb'}, of course, has the same eigenvalues as $\Tb(R)$, but it is non-selfadjoint, and the existing perturbation theorems for non-selfadjoint operators are not applicable since they  concern not eigenvalues but singular   numbers of operators, moreover, it is known that even a very weak perturbation of a non-selfadjoint operator may drastically change characteristics of the spectrum.

The same kind of complication arises when we consider the problem of finding the asymptotics of singular numbers of $\Tb(R)$ for a non-Hermitian matrix $R$. Here, the product-inversion transformation fails even earlier. In fact, while the nonzero eigenvalues of the non-selfadjoint product of two operators do not depend on the order of multiplication,  a similar statement for singular numbers is wrong: after the commutation, we obtain an operator with different singular numbers. Therefore, the reduction to integral operators in $L_{2,\m}$ does not help.  Note that the same obstacle we encounter even in the scalar case, when we are looking for singular numbers asymptotics for   $\Tb(R)$ with a non-real scalar function $R$.

In the  section to follow we are going to present an alternative approach to this problem. 

\section{Eigenvalue asymptotics of restricted operators}\subsection{An abstract theorem} 
Our derivation of formulas for singular numbers and signed asymptotics of eigenvalues follows some ideas implemented  in \cite{MSZLMF}, \cite{SZFA}, however the  specifics of singular measures,  matrix weights, and non-self-adjoint operators require certain modifications. Namely, instead of operators of the form
$\Ab \Bb\Ab$, with bounded $\Bb$ and self-adjoint compact operator $\Ab$, as this is done  in \cite{SZFA}, Sect.5., with various commutation relations, we consider operators which can be written  in the abstract (equivalent) forms
\begin{equation}\label{Operator}
  \Tb(V)=(\g\Ab)^*V(\g\Ab)=(\g\Ab)^*(V\g\Ab)=(V^* \g\Ab)^*(\g\Ab),
\end{equation}
where $\Ab$ is a compact  operator in the Hilbert space $\Hc$, $\g$ is a linear (trace) operator from $\Hc$ to a Hilbert space $\Gc$, bounded on the range of $\Ab$,  and $V$ is a  bounded operator in $\Gc$. It is supposed that  $\Tb(V)$  belongs to the ideal $\Sib_\theta(\Hc)$. In this section, we find conditions granting that 

\begin{equation}\label{mod}
  \Tb(|V|)-|\Tb(V)|\in \Sib_\theta^{\circ},
\end{equation}
and, in the case $V$ is self-adjoint in $\Gc$,
\begin{equation}\label{sign}
  \Tb(V)_{\pm}-\Tb(V_{\pm})\in \Sib_\theta^{\circ}.
\end{equation}

These conditions are formulated in the following way.

Suppose that  there exists a lift $\Vt$ of the operator $V$ from $\Gc$ to the space $\Hc$, as a bounded operator, such that it respects the trace operator $\g$:
\begin{equation*}
  V\g\Ab=\g\Vt\Ab, \, V^*\g\Ab^*=\g\Vt^*\Ab^*,
\end{equation*}
and, simultaneously,
\begin{equation*}
  V^*\g\Ab=\g\Vt^*\Ab, V\g\Ab^*=\g\Vt\Ab^*.
\end{equation*}
We suppose that for the compositions and commutators involving $\Vt$ and $\Ab,$ the following relations hold
\begin{equation}\label{conditions1}
  \g\Ab\in \Sib_{2\theta}; \, \g \Vt\Ab\in\Sib_{2\theta};\, \g\Vt^*\Ab\in\Sib_{2\theta}, \,\g [\Vt,\Ab],\, \g [\Vt^*,\Ab^*]\in \Sib_{2\theta}^{\circ}.
\end{equation}
Additionally, it is supposed that
 there exists a family $Y_\e$, $\e>0$, of positive bounded operators in $\Hc$ such that
\begin{equation}\label{conditions2}
 \bD^{2\theta}( \g(|\Vt|^2-Y_\e^2)\Ab)\le \e,
\end{equation}
and the commutation relations hold:
  \begin{equation}\label{conditions3}
  \g[Y_\e,\Ab]\in \Sib_{2\theta}^{\circ}, \, \g[Y_\e,\Ab^*]\in \Sib_{2\theta}^{\circ}.
\end{equation}
\begin{thm}\label{Thm.separ}Suppose that  conditions \eqref{conditions1}, \eqref{conditions2}, \eqref{conditions3} are satisfied. Then the relation \eqref{mod} is valid. If, additionally, the operator $V$ is self-adjoint in $\Gc,$ then \eqref{sign} holds.
\end{thm}
\begin{proof}
From \eqref{Operator}, it follows that $\Tb(V)^*=(\g\Ab)^*(V^*\g\Ab)$.
Using the lifted operator $\Vt$,   we can represent  $\Tb(V)$  as
\begin{equation}\label{moving V}
 \Tb(V)=(\g\Ab)^*(\g\Vt\Ab)=(\g\Vt^*\Ab)^*(\g\Ab),
\end{equation}
and, for $\Tb(V)^*,$
\begin{equation}\label{moving V*}
 \Tb(V)^*=(\g\Vt\Ab)^*(\g\Ab)= (\g\Ab)^*(\g\Vt^*\Ab).
\end{equation}
Now, consider the operator $\Tb(V)^*\Tb(V)=|\Tb(V)|^2$; using \eqref{moving V}, \eqref{moving V*}, we obtain
\begin{gather}\label{|K|2}
|\Tb(V)|^2=(\g\Ab)^*(\g\Vt^*\Ab)(\g\Vt^*\Ab)^*(\g\Ab).
\end{gather}
For the product of two middle terms in \eqref{|K|2}, we have
\begin{gather*}
(\g\Vt^*\Ab)(\g\Vt^*\Ab)^*=\g((\Vt^*\Ab)(\g\Vt^*\Ab)^*)
=\\\nonumber
\g((\g\Vt^*\Ab)(\Vt^*\Ab)^*)^*=\g(\g(\Vt^*\Ab\Ab^*\Vt))^*.
\end{gather*}
Similarly,
\begin{equation}\label{K^2||}
\Tb(Y_\e)^2=\Tb(Y_\e)\Tb(Y_\e)=(\g\Ab)^*\g(\g Y_\e\Ab\Ab^* Y_\e)^*(\g\Ab),
\end{equation}
with the middle term $\g(\g Y_\e\Ab\Ab^*Y_\e)$ (recall that $Y_\e$ is self-adjoint.)
Our aim now is to use   relations \eqref{conditions1}, \eqref{conditions2}, \eqref{conditions3} to show that  operators \eqref{|K|2} and \eqref{K^2||} differ by a  term, vanishing as $\e\to 0$. The outside factors in these products  are the same, therefore we compare their middle terms.
With this in view, we write
\begin{gather}\label{Vaav1}
\Vt^*\Ab\Ab^*\Vt=(\Vt^*\Ab)(\Ab^*\Vt)= (\Ab\Vt^*+[\Vt^*,\Ab])(\Vt\Ab^*+[\Ab^*,\Vt])=\\\nonumber
\Ab|\Vt|^2\Ab^*+\Ab\Vt^*[\Ab^*,\Vt]+[\Vt^*,\Ab]\Vt\Ab^*+[\Vt^*,\Ab][\Ab^*,\Vt],
\end{gather}
and a similar expansion for the middle terms  in \eqref{K^2||}:
\begin{gather}\label{Vaav2}
  Y_\e\Ab\Ab^*Y_\e =(Y_\e\Ab)(\Ab^*Y_\e)=(\Ab Y_\e+[Y_\e,\Ab])(Y_\e\Ab^*+[\Ab^*,Y_\e])=\\\nonumber
\Ab Y_\e^2\Ab^*+\Ab Y_\e[\Ab^*,Y_\e]+[Y_\e,\Ab]Y_\e\Ab^*+[Y_\e,\Ab][\Ab^*,Y_\e].
\end{gather}

The difference of the first terms in \eqref{Vaav1} and \eqref{Vaav2} equals
\begin{equation*}
 \Hb_\e: =\Ab|\Vt|^2\Ab^*-\Ab Y_\e^2\Ab^*= \Ab(|\Vt|^2-Y_\e^2)\Ab^*, 
\end{equation*}
therefore,
\begin{gather*}
  \g \Hb_\e=\g(\Ab(|\Vt|^2-Y_\e^2)\Ab^*)=(\g\Ab)((|\Vt|^2-Y_\e^2)\Ab^*); \,\\\nonumber 
 (\g\Hb_\e)^*=((|\Vt|^2-Y_\e^2)\Ab^*)^*(\g \Ab)^*=(\Ab (|\Vt|^2-Y_\e^2))
 (\g \Ab)^*=\\\nonumber((|\Vt|^2-Y_\e^2)\Ab)
 (\g \Ab)^*+[\Ab,(|\Vt|^2-Y_\e^2)](\g \Ab)^*,
\end{gather*}
and, finally,
\begin{equation}\label{Vaav5}
\g(\g(\Hb_\e))^*=\left(\g(|\Vt|^2-Y_\e^2)\Ab\right)(\g \Ab)^*+\g\left[\Ab,(|\Vt|^2-Y_\e^2)\right](\g \Ab)^* =I_1+I_2.
\end{equation}

In the first term in \eqref{Vaav5}, due to  condition \eqref{conditions2}, 
$\bD^{2\theta}(\g((|\Vt|^2-Y_\e^2)\Ab))\le \e$, therefore $\bD^{\theta}(I_1)\le C \e.$ For $I_2$, the commutator estimate in \eqref{conditions3} gives $I_2\in \Sib_{\theta}^{\circ}$.

We now check that the remaining terms in \eqref{Vaav1}, \eqref{Vaav2}, the ones involving commutators, produce, after substitution into the middle products in \eqref{|K|2}, \eqref{K^2||},  operators in $\Si_{{\theta}}^\circ.$ This is done in the same manner for all such terms.
Consider, for example, $\Ab\Vt^*[\Ab^*,\Vt]$. Then, 
\begin{equation*}
  (\g \Ab\Vt^*[\Ab^*,\Vt])^*=\left((\g \Ab\Vt^*)[\Ab^*,\Vt]\right)^*=[\Ab^*,\Vt]^*(\g \Ab\Vt^*)^*,
\end{equation*}
 therefore, 
 \begin{equation*}
 \g(\g \Ab\Vt^*[\Ab^*,\Vt])^*=(\g[\Ab^*,\Vt]^*)(\g \Ab\Vt^*)^*
 \end{equation*}
and the corresponding term in $\Tb(V)^*\Tb(V)$ equals
\begin{equation}\label{K[]}
 (\g\Ab)^*(\g[\Ab^*,\Vt]^*)(\g \Ab\Vt^*)^*(\g\Ab).
\end{equation}
In this product, all terms except the second one, belong to $\Sib_{2\theta}$, and the second one, by our conditions, belongs to $\Sib_{2\theta}^\circ$. Thus, the product belongs to $\Sib_{{\theta/2}}^\circ$. A similar reasoning takes care of all remaining terms. This proves that
\begin{equation}\label{Ye}
  \bD^{\theta/2}(|\Tb(\Vt)|^2-\Tb(Y_\e)^2)\le C\e.
\end{equation}
 We recall the result in \cite{BKS}: for non-negative compact operators $\Tb_1,\Tb_2$, 
the property $\bD^p(\Tb_1^2-\Tb_2^2)\le c\e$ implies $\bD^{2p}(\Tb_1-\Tb_2)\le c'\e^{\frac12}$. We apply this result to \eqref{Ye}, with $\Tb_1=|\Tb(V)|,\, \Tb_2=\Tb(Y_\e), $ $p=\theta/2$; this gives us
\begin{equation}\label{Ye2}
  \bD^{\theta}(|\Tb(\Vt)|-\Tb(Y_\e))\le C\e^{\frac1\theta}.
\end{equation}
On the other hand, by the condition \eqref{conditions2}, 
\begin{equation}\label{Ye2t}
  \bD^{\theta}(\Tb(Y_\e)-\Tb(|\Vt|))=\bD^{\theta}(\Tb(Y_\e-|\Vt|))\le C \e^{\frac1\theta}.
\end{equation}
Due to the arbitrariness of $\e$, \eqref{Ye2}, \eqref{Ye2t} imply $\bD^{\theta}(\Tb(|V|)-|\Tb(V)|)=0,$ which proves the  statement in  \eqref{mod}.
  As for   \eqref{sign}, for a  self-adjoint operator $V$,
$\Tb(V)_{\pm}=\frac12(|\Tb(V)|\pm \Tb(V))$, and the result follows from the already established closeness  of  $|\Tb(V)|$ and $\Tb(|V|)$:
\begin{gather}\label{hermite}
(\Tb(V))_+=\frac12(|\Tb(V)|+\Tb(V))=
\frac12\left(|\Tb(V)|-\Tb(|V|)\right)+\\\nonumber
\frac12\left(\Tb(|V|)+\Tb(V)\right)=
\frac12\left(|\Tb(V)|-\Tb(|V|)\right)+\Tb(V_+),
\end{gather}
with the first term in $\Sib_{\theta}^{\circ}$,
\end{proof}
\subsection{Concrete approximation}
We are returning to our spectral problem for the operator $\Tb(V)$.  As in the scalar case, using Lemma \ref{lem.appr}, we can suppose that the matrix-valued density $V(x),\, x\in\Si$ admits an extension $\Vt\in C^{\infty}_0$ to the whole space $\R^\Nb$. We need to check the conditions in \eqref{conditions1}, \eqref{conditions2}, \eqref{conditions3}. The first two estimates in \eqref{conditions1} are valid due to Theorem \ref{Est.Thm}. The commutator estimates  follow from the fact that for a smooth matrix-function $\Vt(x)$, the commutator of $\Vt$ and the order $-l$ pseudodifferential operator $\Ab$ is a pseudodifferential operator of order $-l-1$, and Theorem \ref{Est.Thm} gives the required estimate.  Further on, we accept 
$Y_\e(x)=(|\Vt(x)|^2+\e^2)^{\frac12}=(\Vt(x)^*\Vt(x)+\e^2)^{\frac12};$
 it is a smooth function and it serves as a good smooth approximation to $|\Vt|(x)$; conditions \eqref{conditions2}, \eqref{conditions3} follow.

\subsection{Operators on Lipschits surfaces}\label{results} Here  we describe the concrete  setting where the abstract scheme above is applied.  As explained in Sect. \ref{matrix breakdown}, for a compact Lipschitz surface $\Si$ of dimension $s<\Nb$ (and codimension $d=\Nb-s>0$), the role of $\m$ is played by the natural Hausdorff measure on $\Si$, equivalently, the measure  induced by the Lebesgue measure in $\R^\Nb$.  Let $\Ab_0$ be a classical scalar order $-l<0$ pseudodifferential operator,  which we can suppose to be supported in a domain
$\Om\supset \Si$,  $\Ab=\Ab_0\otimes\Eb_m$.  It is known that $\Ab^*\Ab$ is an integral, potential type, operator in $\R^\Nb$ with kernel having weak singularity at the diagonal. 
This kernel has principal part  positively homogeneous of order $\rb=2l-\Nb$, if $\rb$  is not a nonnegative integer, or may contain  $\log$-factor  otherwise.  We do not need explicit expressions for this kernel, since the asymptotics of the spectrum is conveniently expressed via the symbol $\ab_{-l}(x,\x)$, see \cite{RT}, \cite{RTInt}.   Since $\Ab$ is a bounded operator from $L_{2}(\R^\Nb) $
 to the Sobolev space $\Ht^l(\R^\Nb)$, the restriction $\g:\Ht^l(\R^\Nb)\to L_2(\Si)$ is well defined, see Sect.2. For a matrix function $V(x)\in L_{(\theta),\m}$, the operator $\Tb(\Ab, V)$ is defined by the quadratic form $\int\langle V(x)\g\Ab u, \g\Ab u\rangle\m(dx),$ $ u\in L_2(\R^\Nb)$.
A bounded density $\ro(x)$ $x\in\Si,$ is determined by the principal symbol of $\Ab$ and the surface $\Si$, defined by a complicated algorithm, which we do not reproduce here, see Theorem 6.2 in \cite{RT} (the particular expression of $\ro(x)$ is of no importance here). As explained in Sect.2, the results of \cite{RS}, \cite{RT} take care of the case when the matrix $V(x)$
is sign-definite, say, nonnegative, belongs to $L_{(\theta),\m}(\Si)$, and justify  the Weyl asymptotics of eigenvalues
\begin{equation}\label{Weyl}
\nb^{\theta}(\Tb(V))=\int_{\Si}\ro(x)\Tr(V(x)^{\theta})dx.
\end{equation}
 Now, we can use Theorem \ref{Thm.separ}, which, being applied to this concrete situation, gives the following results.
First, let the matrix $V(x)$ be Hermitian for all $x\in\Si$. We approximate it by a smooth matrix $\Vt$ using Lemma \ref{lem.appr}, and due to Lemma \ref{BSlem}, it suffices to prove  the asymptotic  formulas for this approximation.
Then, according to \eqref{mod}, the asymptotic characteristics of positive eigenvalues of $\Tb(V)$ coincide with asymptotic characteristics of eigenvalues of $\Tb(V_{+})$, ditto for negative eigenvalues-- and for those, the asymptotic formulas have been already proved.  We apply \eqref{Weyl}, which  give
\begin{equation*}
  \nb^{\theta}_{\pm}(\Tb(V))=\nb^{\theta}_{\pm}(\Tb(V_{\pm}))=\int_{\Si}\Tr(V_{\pm}(x)^{\theta})\mu(dx).
\end{equation*}
Next, for an \emph{arbitrary matrix} density $V(x)$, we apply \eqref{sign} in Theorem \ref{Thm.separ}, according to which, asymptotic characteristics of the singular numbers of $\Tb(V)$ coincide with the asymptotic characteristics of singular numbers, this means, eigenvalues of the non-negative operator $\Tb(|V|)$. Again, we apply \eqref{Weyl} to the latter operator, which gives
\begin{equation*}
\nb^{\theta}(\Tb(V))=\nb^{\theta}(\Tb(|V|))=\int_{\Si}\Tr(|V(x)|^{\theta})\mu(dx)
\end{equation*}
These formulas prove our Theorem \ref{Main theorem pm} in Sect.2.

\section{Lower estimates} While the eigenvalue asymptotics is established only for measures being the surface measure on Lipschitz surfaces or their unions, it is possible to show that the  order in the eigenvalue estimates in Sect. \ref{Sub.estim}, as well as the singular numbers estimates are order sharp. This kind of lower  estimates for the case of a positive weight $V$ have been declared in \cite{RT}, with reference to  more general  results in \cite{Tr}, Theor. 27.15 and Theor. 28.6. There, it was supposed that the real-valued weight function $V$ is greater than some positive constant, while the operator $\Ab$ was a parametrix of an elliptic  operator.  Here, we can prove a somewhat more general result. We show that,  in the Hermitian matrix  case, under some 
mild  condition imposed on the positive (negative) part of $V(x)$, the order of the upper estimate for positive(negative) eigenvalues of $\Tb(V)$ is sharp, namely, this is supported by a lower estimate of the same power order. Further on, for a not necessarily Hermitian matrix $V(x)$, again, under  some mild regularity condition, we show that the upper estimate for the singular values of $\Tb(V)$ is order sharp, it is supported by  the lower estimate with the same power order.
Such results were not known even in the scalar case,
 \begin{thm}\label{Thm.lower}Let the measure $\mu$ belong to $\AF^{s}$, $0<s\le\Nb$, $\Si=\supp(\mu)$; let the pseudodifferential order $-l$ operator  $\Ab=\Ab_0\otimes \Eb_{m}$ in $\R^\Nb$ be elliptic in some domain $\Om\subset \R^\Nb$ with non-empty intersection with $\Si$;  let $V$ be a matrix density, $V\in L_{(\theta),\mu}(\Si)$, where, as always, $\theta=\frac{s}{s-\Nb+2l}.$ Consider the operator $\Tb=\Tb(V)$. Suppose that $|V(x)|\ge h>0$ in $\Si\cap \Om$.
 \begin{equation}\label{lower.estim}
{\bd^{\theta}}(\Tb)\equiv\lim\inf_{\la\to0} \la^{\theta} n(\la, \Tb)>0;
 \end{equation}
 if, moreover, the matrix $V(x)$ is Hermitian, and $V(x)_+\ge c>0$ for $x\in \Si\cap \Om$, then
 \begin{equation}\label{Lower.estim.positive}
   \bd^{\theta}_+(\Tb)>0,
 \end{equation} 
 and the same for $ \bd^{\theta}_-(\Tb)$, provide $V(x)_-\ge C>0$ in $\Si\cap \Om$. 
  \end{thm}
  \begin{proof} The second statement in the Theorem is a consequence of the  first one, and we present the proof of the latter.

We start by recalling the result  in \cite{Tr}, Theor. 27.15 and Theor. 28.6,
where the scalar case with positive function $V$ is considered, and with $\Ab$ being the parametrix for some elliptic order $l$ operator in $\Om$. The proof in \cite{Tr} of the lower estimate (in our terms)
\begin{equation*}
  \nb^{\theta}(\Tb(V))>0,
\end{equation*}
 goes as follows. For a given $\de>0,$ a system $\U(\de)$ of no less than $c\de^{-s}$ disjoint balls with radius $\de$ can be constructed, so that for each ball $B\in\U(\de),$ the portion of $\mu$ in the concentric, twice smaller ball $B'$, is greater than $c\de^{s}$. Then for each ball $B\in\U(\de)$, a function $\f_{B}\in C_0^{\infty}(B)$ is constructed, which equals $1$ on $B'$ and has $\|\Ab\f_{B}\|_{L_2}\ge C \de^{-l}$ -- this is a consequence of the ellipticity of $\Ab$. After that, it is shown  (see (28.43) in \cite{Tr}) that on the $C \de^{-s}-$ dimensional subspace $\Lc(\de)$ in $L_2(\R^\Nb)$, spanned by the functions  $\f_B,$ $B\in\U(\de)$,
\begin{equation*}
\int|(\Ab u)(x)|^2\mu(dx)\ge C\de^{2l-\Nb+s}\|u\|_{L_2}^2,\, u\in\Lc(\de).
\end{equation*}
From the variational principle, by setting $\de=\la^{\frac{1}{s}}$,  this implies that $n(\la,\Tb(V))\ge C \la^{-\theta}$.

It is easy to see, that all the steps in the reasoning above carry over without complications to the  matrix case, with    $V(x)\ge c>0$, $x\in\Si\cap\Om$ being a positive definite matrix function; the balls in the construction should be chosen contained in $\Om,$ for sufficiently small $\la$.

  Now we consider the case of a non-Hermitian  matrix weight $V(x).$ Again, to obtain a lower estimate for the counting function of singular numbers of $\Tb(V),$ it suffices, by Lemma \ref{BSlem}, to obtain lower estimates for the asymptotic coefficients $\bd^{\theta}(\Tb(V_\e))$, not depending on $\e$, for densities $V_\e$ approximating $V$ in the metric of $L_{(\theta),\m}$. We 
 consider such an approximating function $\Vt,$ constructed according to Lemma \ref{lem.appr}. If we suppose  that $|V(x)|>c>0$ in $\Si\cap\Om$, the approximating matrix-function $\Vt(x)$ can be chosen with the same property (probably, with a slightly smaller $\Om$) with lower estimate for $\bd^{\theta}(\Tb(V_\e))$ independent on $\e$.

Now, we can apply Theorem \ref{Thm.separ}.
Since $\Tb(|\Vt|)-|\Tb(\Vt)|\in\Si_{\theta}^0,$
it follows that $\bd^{\theta}(|\Tb(\Vt)|)=\bd^{\theta}(\Tb(|\Vt|))$. For the last quantity, with a non-negative  matrix $|\Vt|$,  the lower estimate is already known, and this proves the first part of Theorem. The second part, for a Hermitian matrix, just as previously, follows from the first one, due to the relation \eqref{hermite}. This proves Theorem \ref{Thm.lower}.
 \end{proof}

\end{document}